\documentclass{amsart}

\usepackage{amsmath}
\usepackage{amsthm}
\usepackage{amsopn}
\usepackage{amssymb}
\usepackage[all]{xy}

\newcommand{\mc}[1]{\mathcal{#1}}

\newcommand{\mr}[1]{\mathrm{#1}}
\newcommand{\norm}[1]{\lVert #1 \rVert}

\newcommand{\br}[1]{\overline{#1}}
\newcommand{\ul}[1]{\underline{#1}}

\newcommand{\td}[1]{\widetilde{#1}}

\newcommand{\ZZ}{\mathbb{Z}}

\newcommand{\CC}{\mathbb{C}}
\newcommand{\QQ}{\mathbb{Q}}

\newcommand{\FF}{\mathbb{F}}
\newcommand{\GG}{\mathbb{G}}
\newcommand{\MS}{\mathbb{S}}
\newcommand{\AF}{\mathbb{A}}

\theoremstyle{definition}
 \newtheorem{thm}[equation]{Theorem}
 \newtheorem{cor}[equation]{Corollary}
 \newtheorem{lem}[equation]{Lemma}
 \newtheorem{prop}[equation]{Proposition}

 \newtheorem{rmk}[equation]{Remark}
\newtheorem*{thm*}{Theorem}
\newtheorem*{cor*}{Corollary}
\newtheorem*{lem*}{Lemma}
\newtheorem*{prop*}{Proposition}
\newtheorem*{defn*}{Definition}
\newtheorem*{ex*}{Example}
\newtheorem*{exs*}{Examples}
\newtheorem*{rmk*}{Remark}
\newtheorem*{claim*}{Claim}

\numberwithin{equation}{section}
\numberwithin{figure}{section}
\DeclareMathOperator{\Ext}{Ext}

\DeclareMathOperator{\Map}{Map}
\DeclareMathOperator{\TMF}{TMF}
\DeclareMathOperator{\Gal}{Gal}
\DeclareMathOperator{\End}{End}
\DeclareMathOperator{\Aut}{Aut}
\DeclareMathOperator{\Fr}{Fr}
\DeclareMathOperator{\ord}{ord}

\DeclareMathOperator*{\hocolim}{hocolim}
\DeclareMathOperator*{\colim}{colim}
\DeclareMathOperator*{\Tot}{Tot}

\title[Congruences given by the $\beta$ family]{Congruences between modular forms given by the divided $\beta$ family in
homotopy theory}
\address{
Department of Mathematics \\
Massachusetts Institute of Technology \\
Cambridge, MA 02139, U.S.A.}
\author[Mark Behrens]{Mark Behrens$\sp 1$}

\date{\today}

\begin{document}

\begin{abstract}
We characterize the $2$-line of the $p$-local Adams-Novikov spectral
sequence in terms of modular forms satisfying a certain explicit congruence
condition for primes $p \ge 5$.  
We give a similar characterization of the $1$-line,
reinterpreting some earlier work of A.~Baker and G.~Laures.  
These results are then used to
deduce that, for $\ell$ a prime which generates $\ZZ_p^\times$, the
spectrum $Q(\ell)$ detects the $\alpha$ and $\beta$ families in the stable
stems.
\end{abstract}

\maketitle

\footnotetext[1]{The author was supported by the NSF grant DMS-0605100, 
the Sloan Foundation, and DARPA.}

\tableofcontents
\bibliographystyle{amsalpha}

\section{Introduction}\label{sec:intro}

The Adams-Novikov spectral sequence 
$$ \Ext^{s,t}_{BP_*BP}(BP_*, BP_*) \Rightarrow (\pi^S_{t-s})_{(p)} $$
is one of the main tools for organizing
periodic phenomena in the $p$-local stable homotopy groups of spheres.
Assuming that $p$ is an odd prime, the $1$-line is generated by elements
$$ \alpha_{i/j} \in \Ext^{1,2(p-1)i-1}_{BP_*BP}(BP_*, BP_*) $$
of order $p^j$, for $i \ge 1$ and $j$ satisfying 
$$ j = \nu_p(i) + 1. $$
The elements $\alpha_{i/j}$ are all permanent cycles, and detect the
generators of the image of the $J$-homomorphism.
The image of $J$ admits a global description in terms of denominators of
Bernoulli numbers: there is a correspondence
$$ \alpha_{i/j} \leftrightarrow B_{t} $$
between the generator $\alpha_{i/j}$ and the $t^\mr{th}$ Bernoulli number for $t
= (p-1)i$.
The order $p^j$ of the element $\alpha_{i/j}$ is equal to the $p$-factor of
the denominator of the quotient
$$ \frac{B_{t}}{t}. $$
Thus the $1$-line of the Adams-Novikov spectral sequence is
governed by the $p$-adic valuations of the denominators of the Bernoulli
numbers.
The purpose of this paper is to provide a similar description for the
$2$-line of the Adams-Novikov spectral sequence, in terms of certain
congruences of modular forms.

Let 
$$ M_k(\Gamma_0(N)) $$
denote the space of weight $k$ modular forms for $\Gamma_0(N)$ defined
over $\ZZ$.  For a ring $R$, let
\begin{equation}\label{eq:Mdef} 
M_k(\Gamma_0(N))_R = M_k(\Gamma_0(N)) \otimes R
\end{equation}
be the corresponding space of modular forms defined over $R$.  
If $N = 1$, we shall simplify the notation:
$$ (M_k)_R := M_k(\Gamma_0(1))_R. $$
We shall sometimes work with modular forms which are simply meromorphic at
$\infty$, which we shall denote
$$ M_k(\Gamma_0(N))_R^0 = M_k(\Gamma_0(N))_R[\Delta^{-1}] $$
where $\Delta \in M_{12}$ is the discriminant.

\begin{rmk}
Implicit in our definition of the notation $M_k(\Gamma_0(N))_R$ given by
(\ref{eq:Mdef}) is a non-trivial base change theorem.  One typically
requires $N$ to be invertible in $R$, and then one
regards the modular forms for a ring $R$ as the sections of a certain line
bundle of the base-change of the moduli stack of elliptic curves to $R$.
In most instances considered in this paper, these two notions agree, see
\cite[1.7, 1.8]{Katz}.
\end{rmk}

The
$q$-expansion gives rise to an embedding
\begin{align*}
M_k(\Gamma_0(N)) & \hookrightarrow \ZZ[[q]] \\
f & \mapsto f(q)
\end{align*}
and consequently embeddings 
\begin{gather*}
M_k(\Gamma_0(N))_R \hookrightarrow R[[q]], \\
M_k(\Gamma_0(N))_R^0 \hookrightarrow R((q)).\\
\end{gather*}
Therefore, a modular form over $R$ is determined by its weight and its 
$q$-expansion.

For any $f \in (M_k)_R$, and any prime $\ell$, the power series
$$  (V_\ell f)(q) := f(q^\ell) $$
is the $q$-expansion of a modular form
$$ V_\ell f \in M_k(\Gamma_0(\ell))_R. $$

Suppose that $p$ is a prime greater than $3$. 
Miller, Ravenel, and Wilson showed that the $2$-line of the Adams-Novikov
spectral sequence
is generated by elements 
$$ \beta_{i/j,k} \in \Ext^{2,*}_{BP_*BP}(BP_*, BP_*) $$ 
for $i$, $j$, and $k$ satisfying
certain elaborate conditions (see Theorem~\ref{thm:MillerRavenelWilson}).  
Our main theorem is stated below.

\begin{thm}\label{thm:main}
For each additive generator 
$$ \beta_{i/j,k} \in \Ext^{2, *}_{BP_*BP}(BP_*, BP_*) $$ 
there is an associated modular form
$$ f_{i/j,k} \in M_{t} $$
(where $t = i(p^2-1)$)
satisfying:
\begin{enumerate}
\item The $q$-expansion $f_{i/j,k}(q)$ is not congruent to $0$ mod $p$.
\item We have $\ord_q f_{i/j,k}(q) > \frac{t-j(p-1)}{12}$ or $\ord_q
f_{i/j,k}(q) = \frac{t-j(p-1)-2}{12}$.
\item There does not exist a form 
$$ g \in M_{t'}, \qquad \text{for $t' < t$}, $$ 
satisfying
$$ f_{i/j,k}(q) \equiv g(q) \mod p^k. $$
\item For every prime $\ell \ne p$, there exists a form 
$$ g_\ell \in M_{t - j(p-1)}(\Gamma_0(\ell)) $$
satisfying 
$$ f_{i/j,k}(q^\ell) - f_{i/j,k}(q) \equiv  g_\ell(q) \mod p^k. $$
\end{enumerate}
\end{thm}

The congruence conditions met by the forms $f_{i/j,k}$ are sharp; we have
the following converse theorem.

\begin{thm}\label{thm:converse}
Suppose there exists a modular form
$$ f \in M_{t} $$
satisfying Conditions \ref{thm:main}(1)-(4), where $t \equiv 0 \mod
(p-1)p^{k-1}$.  Then
$t = i(p^2-1)$ for some $i$, and 
if $i$ is not a power of $p$,  
there is a corresponding generator
$$ \beta_{i/j,k} \in \Ext^{2,*}_{BP_*BP}(BP_*, BP_*). $$
\end{thm}

Finally, the congruence condition given in Theorem~\ref{thm:main} exhibits a
certain rigidity, as explained in the following theorem.

\begin{thm}\label{thm:rigid}
Suppose that $\ell_0$ is a prime which generates $\ZZ_p^\times$.  Then, if
$f$ is a modular form of weight $t \equiv 0 \mod (p-1)p^{k-1}$ 
satisfying Conditions~\ref{thm:main}(1)-(3),
and Condition \ref{thm:main}(4) for $\ell = \ell_0$, then $f$ satisfies
Condition \ref{thm:main}(4) for \emph{all} primes $\ell \ne p$.
\end{thm}

\begin{rmk}
In \cite{Laures}, G.~Laures introduced the $f$-invariant, a 
higher analog of the Adams
$e$-invariant, which gives an injection of the groups
$\Ext^{2,*}_{BP_*BP}(BP_*, BP_*)$ into groups closely related to Katz's
ring of divided congruences.  Laures' $f$-invariant 
associates to an element $\beta_{i/j,k}$, a higher congruence between
modular forms.  It is natural to ask what relation this congruence has to
the congruences given by condition~(4) of Theorem~\ref{thm:main}.
Unfortunately, the author is unable to determine the relationship between
these two approaches.  Theorem~\ref{thm:converse} does indicate that the
congruences we are studying give a precise description of the
$\beta$-family in terms of modular forms, whereas the $f$-invariant, while
injective, has a very small image in the ring of divided congruences.  The
image of the elements $\beta_i$ under the $f$-invariant, however, has been
characterized by J.~Hornbostel and N.~Naumann \cite{HornbostelNaumann}.
\end{rmk}

The proofs of Theorems~\ref{thm:main}-\ref{thm:rigid} 
use the spectrum $Q(\ell)$ introduced by the
author in \cite{Behrens}, \cite{Behrensmodular}.  Analyzing
the chromatic spectral
sequence 
$$ E_1^{s,t} = \pi_t M_s Q(\ell) \Rightarrow \pi_{t-s} Q(\ell), $$
we observe that a density
result \cite{BehrensLawsondense} relates part of the $2$-line of the chromatic
spectral sequence of $Q(\ell)$ to the
$2$-line of the chromatic spectral sequence for the sphere.  We also
analyze the $0$ and $1$-lines of the chromatic spectral sequence of
$Q(\ell)$, and find:
\begin{enumerate}
\item The $0$-line $\pi_t M_0Q(\ell)$ is concentrated in $t = 0, -1, -2$
(Corollary~\ref{cor:M0Ql}).

\item The $1$-line $\pi_t M_1Q(\ell)$ is generated in degrees congruent to
$0 \mod 4$ by the images of the elements $\alpha_{i/j} \in \pi_* M_1S$
(Corollary~\ref{cor:M1Ql}).
\end{enumerate}

In fact, the additive generators of $\pi_{4t}Q(\ell)$ are given by the
Eisenstein series $E_{2t} \in (M_{2t})_\QQ$ (\ref{prop:eisenstein}), 
and the orders of the groups
$\pi_{4t}M_1Q(\ell)$ are directly linked to the $p$-adic valuation of the
denominators of the Bernoulli numbers $\frac{B_{2t}}{4t}$ through the
appearance of the Bernoulli numbers in the $q$-expansions of the Eisenstein
series.  The relationship was originally made precise by G.~Laures
\cite{Lauresthesis}, and rederived by A.~Baker in \cite{Baker}, 
where Hecke operations are used 
to conclude that
Eisenstein series generate the $1$-line of the $\TMF$-Adams Novikov
spectral sequence.  Our analysis is closely related to these.

As a consequence of our study of the chromatic spectral sequence for
$Q(\ell)$ we are able to prove the following theorem.

\begin{thm*}[Theorem~\ref{thm:greek}]
The images of the elements $\alpha_{i/j}$ and the elements $\beta_{i/j,k}$ 
under the homomorphism
$$ \pi_* S_{E(2)} \rightarrow \pi_*Q(\ell) $$
are non-trivial.
\end{thm*}

This theorem shows that 
the homotopy of the spectrum $Q(\ell)$ is closely tied
to Greek letter phenomena.  It also gives credibility to the author's hope
that the following questions have affirmative answers.
\begin{enumerate}
\item Are the homotopy Greek letter elements $\beta^h_{i/j,k}$ (see
\cite{Behrensrootkin}) detected by the spectra $Q(\ell)$ at the primes
$2$ and $3$?
\item Do the spectra $Q_U(K^{p,\ell})$ (constructed using Shimura varieties
of type $U(1,n-1)$ in \cite{BehrensLawson}) detect the $v_n$-periodic Greek
letter elements?
\end{enumerate}

\noindent
{\bf Organization of the paper.}

In Section~\ref{sec:css} we summarize the chromatic spectral sequence.  We
also recall Morava's change of rings theorem, which relates the terms of
the chromatic spectral sequence to the cohomology of the Morava stabilizer
groups $\MS_n$.

In Section~\ref{sec:tmf} we explain how to associate a $p$-complete 
$\TMF$-spectrum to
every compact open subgroup of the adele group $GL_2(\AF^{p,\infty})$,
following standard conventions used in the theory of modular forms.
Certain $E_\infty$-operations between these spectra are given by elements
of $GL_2(\AF^{p,\infty})$.

In Section~\ref{sec:Q(S)} we explain how to use the $GL_2$ action of
Section~\ref{sec:tmf} to define spectra $Q(S)$ for a set of primes $S$.
These spectra agree with the spectra $Q(\ell)$ defined by the author in
\cite{Behrens}, \cite{Behrensmodular}
when $S = \{ \ell \}$.  The approach of this paper, however, mimics that of
\cite{BehrensLawson}.  We explain how the results of \cite{Behrens},
\cite{BehrensLawsondense} arise in this framework.  In particular, 
we identify the
$K(2)$-localization of $Q(S)$ as the homotopy fixed points of a dense
subgroup $\Gamma_S$ of the Morava stabilizer group $\MS_2$.

In Section~\ref{sec:building} we explain how the building resolution of
$Q(\ell)$ given in \cite{Behrens} can be recovered using the methods of
\cite{BehrensLawson}.  We use this resolution to define a finite cochain
complex $C^\bullet(\ell)$ of modular forms whose cohomology gives
$\pi_*Q(\ell)$.

In Section~\ref{sec:coface} we express the first differential in
$C^\bullet(\ell)$ in terms of the Verschiebung of modular forms.

In Section~\ref{sec:cssQl} we describe the chromatic spectral sequence of
$Q(\ell)$.  In particular, we show that its $E_1$-term consists of three
lines: $M_0Q(\ell)$, $M_1 Q(\ell)$, and $M_2 Q(\ell)$.  We explain how to
modify the chain complex $C^\bullet(\ell)$ to compute these terms. 
We also relate $M_2 Q(\ell)$ to the cohomology of the group $\Gamma_\ell$.

In Section~\ref{sec:M0Ql} we show that $\pi_t M_0 Q(\ell)$ is concentrated
in $t = 0, -1, -2$.  We also deduce that the rest of $\pi_t Q(\ell)$ is
$p$-torsion, and give bounds for the torsion.

In Section~\ref{sec:M1Ql} we compute $\pi_{4t} M_1 Q(\ell)$, and show that
its generators can be expressed as Eisenstein series.  The orders of these
groups are given by the $p$-adic valuation of the Bernoulli numbers
$B_{2t}/{4t}$.

In Section~\ref{sec:congruences} we recall theorems of Serre and
Swinnerton-Dyer, which
relate congruences amongst $q$-expansions of modular forms to
multiplication by the Hasse invariant $E_{p-1}$.

In Section~\ref{sec:M2Ql} we analyze $\pi_* M_2Q(\ell)$, and prove
Theorems~\ref{thm:main} -- \ref{thm:rigid}.

In Section~\ref{sec:greek} we deduce Theorem~\ref{thm:greek} from our
extensive knowledge of the chromatic spectral sequence for $Q(\ell)$.

\noindent
{\bf Acknowledgements}
This paper was prepared while the author visited Harvard University, and
the author is grateful for their hospitality.  The author is also grateful
to Niko Naumann for pointing out an error in an earlier draft concerning
the integrality of Eisenstein series.

\section{The chromatic spectral sequence}\label{sec:css}

Let $X$ be a spectrum.  The chromatic tower for $X$ is the tower of
Bousfield localizations with respect to the Johnson-Wilson spectra $E(n)$.
$$
\xymatrix{
M_0X \ar@{=}[d] & 
M_1X \ar[d] &
M_2X \ar[d] &
\\
X_{E(0)} &
X_{E(1)} \ar[l] &
X_{E(2)} \ar[l] &
\quad \ar[l] &
\cdots
}
$$
Here the fibers $M_nX$ are the $n$th monochromatic layers.  They admit a
presentation as
\begin{equation}\label{eq:Mncolimit}
M_nX \simeq \hocolim_I M(I)^0 \wedge X_{E(n)}
\end{equation}
where $M(I)^0 = M(i_0, \ldots, i_{n-1})^0$ is the generalized Moore
spectrum with top cell in dimension $0$ and $BP$-homology, 
$$ BP_* M(i_0, \ldots, i_{n-1}) = \Sigma^{-\norm{I}-n} BP_*/(p^{i_0}, v_1^{i_1}, \ldots,
v_{n-1}^{i_{n-1}}) $$
with
$$ \norm{I} := \sum_{j} 2i_j(p^j-1) $$
and $I$ ranges over a suitable
cofinal collection of multi-indices.  (The existence of such a system of
generalized Moore spectra is a consequence of the Hopkins-Smith 
Periodicity Theorem \cite{HopkinsSmith}.)

Applying homotopy to the chromatic tower yields the chromatic spectral sequence
$$ 
E_1^{n,k} = \pi_k M_n X \Rightarrow \pi_{k-n} X_{(p)}
$$
which is conditionally convergent if $X$ is harmonic.

Morava's change of rings theorem states that the Adams-Novikov spectral
sequence for $M_n X$ takes the form
\begin{equation}\label{eq:MnANSS}
H^s_c(\MS_n, (M_nE_n)_t(X))^{\Gal(\FF_p)} \Rightarrow \pi_{t-s}(M_n X).
\end{equation}

For $p \gg n$ this spectral sequence is known to collapse.  A simple
instance of this (for $X = S$) is given by the following lemma.

\begin{lem}
For $X = S$ and 
$2p-2 \ge \max \{ n^2, 2n+2 \} $, the spectral sequence (\ref{eq:MnANSS})
collapses:  the groups
$$
H^s_c(\MS_n, \pi_t M_nE_n)^{\Gal(\FF_p)}
$$
are zero unless $t \equiv 0 \mod 2(p-1)$. 
\end{lem}

\begin{proof}
The action of an element $a$ of 
the finite subgroup $\FF_{p}^\times \subset \MS_n$
on 
$$ \pi_{2t} M_n E_n \cong (\pi_{2t} E_n)/(p^\infty, v_1^\infty, \ldots,
v_{n-1}^\infty).
$$
is given by multiplication by $\td{a}^t$, where $\td{a}$ is the image of
$a$ under the Teichm\"uller embedding $\FF^\times_{p} \subset \ZZ_p^\times$.  
Since $\FF_p^\times$ is cyclic of order $p-1$, it follows that $\FF_{p-1}$
acts trivially if and only if $t \equiv 0 \mod p-1$.  Because the subgroup
$\FF_{p}^\times \subset \MS_n$ is central and Galois invariant, it follows
that there is an induced action of $\FF_p^\times$ on 
\begin{equation*}
H^s_c(\MS_n, \pi_{2t}M_n E_n)^{\Gal(\FF_p)}
\end{equation*}
However, the induced action on cohomology must be trivial, because the
action is obtained by restriction from the action of $\MS_n$.  Thus, the
cohomology groups must be trivial, except when $t \equiv 0 \mod p-1$.  The
result follows from the fact that if $n < p-1$, the group $\MS_n$ has
cohomological dimension $n^2$. 
\end{proof}

The sparsity of the spectral sequence (\ref{eq:MnANSS}), together with the
fact that $E_2^{s,t} = 0$ for $s \le n^2$ implies the
following corollary

\begin{cor}\label{cor:collapseisomorphism1}
For $2(p-1) \ge \max \{ n^2, 2(n+1) \} $, we have
$$ 
\pi_t M_n S \cong H^s_c(\MS_n, \pi_{t+s} M_n E_n)^{\Gal(\FF_p)}
$$
where $t = 2k(p-1) - s$ and $0 \le s < 2(p-1)$.
\end{cor}

\section{Adelic formulation of $\TMF$}\label{sec:tmf}

Let $\AF$ denote the rational adeles.  For a set of finite places $S$ of
$\QQ$, define
\begin{align*}
\widehat{\ZZ}^{S} & := \prod_{\ell \notin S} \ZZ_\ell, \\
\AF^{S,\infty} & := \widehat{\ZZ}^S \otimes \QQ.
\end{align*}
Fix a rank $2$ module:
$$ V^p := \AF^{p,\infty} \oplus \AF^{p,\infty}, $$
and let $L^p = \widehat{\ZZ}^p \oplus \widehat{\ZZ}^p$ be the canonical
lattice contained in $V^p$.

For an elliptic curve $C$ over an algebraically closed field $k$ of
characteristic unequal to $\ell$, let
$$ T_\ell(C) := \varprojlim_k C(k)[\ell^k] $$
denote the $\ell$-adic Tate module of $C$.  The Tate module $T_\ell(C)$ is
a free $\ZZ_\ell$-module of rank
$2$.  
If the characteristic of $k$
is zero or $p$, the $\ell$-adic Tate modules
assemble to give the $\AF^{p,\infty}$-module
$$ V^p(C) := T^p (C) \otimes \QQ, $$
where
$$ T^p(C) := \prod_{\ell \ne p} T_\ell(C). $$
There is a canonical short exact sequence
$$ T^p(C) \rightarrow V^p(C) \xrightarrow{u} C[tor^p] $$
where $C[tor^p]$ is the subgroup of the group of $k$-points of $C$
consisting of those points which are torsion of order prime to $p$.

A \emph{rational uniformization} is an isomorphism
$$ \eta : V^p \xrightarrow{\cong} V^p(C). $$
The group $GL_2(\AF^{p,\infty})$ acts on the set of rational
uniformizations by precomposition.

As explained in \cite[Sec.~3.2]{BehrensLawson}, 
a rational uniformization $\eta$ of $V^p(C)$ gives rise to a 
prime-to-$p$ quasi-isogeny
\begin{equation}\label{eq:quasiisogeny}
\phi_{\eta}: C \rightarrow C_\eta
\end{equation}
(up to isomorphism of $C_\eta$).
If the uniformization has the property that 
\begin{equation}\label{eq:Tpcontainment}
T^p(C) \subseteq \eta(L^p),
\end{equation}
the quasi-isogeny is an isogeny.  In this case,
the (isomorphism class of the) isogeny $\phi_\eta$ is characterized 
by its kernel $H_\eta$, which is given by:
$$ \ker(\phi) = H_\eta := 
\mr{image}(L^p \hookrightarrow V^p \xrightarrow{\eta} V^p(C) 
\xrightarrow{u} C[tor^p]). $$
(The case of more general $\eta$, not satisfying (\ref{eq:Tpcontainment}),
is easily generalized from this, producing quasi-isogenies 
$\phi_\eta$ which need not be isogenies.)

For a subgroup
$$ K^p \subset GL_2(\AF^{p,\infty}), $$
we let $[\eta]_{K^p}$ denote the $K^p$-orbit of rational uniformizations
generated by $\eta$.  The we shall refer to the orbit $[\eta]_{K^p}$ as an
\emph{$K^p$-level structure}.  If we define
$$ K^p_0 := GL_2(\widehat{\ZZ}^p) \subset GL_2(\AF^{p,\infty}), $$
then, given a rational uniformization
$$ \eta : V^p \xrightarrow{\cong} V^p(C), $$
the isomorphism class of the quasi-isogeny $\phi_\eta$ depends \emph{only}
on the $K^p_0$ level structure $[\eta]_{K^p_0}$.

If $C$ is an elliptic scheme over a connected base $S$, we can pick a
geometric point $s$ of $S$ and talk about level structures of the fiber
$C_s$, \emph{provided they are $\pi_1(S,s)$-invariant}.  
Given a $\pi_1(S,s)$-invariant 
$K^p_0$-level structure $[\eta]_{K^p_0}$ represented by a
rational uniformization
$$ \eta: V^p \xrightarrow{\eta} V^p(C_s), $$
(satisfying (\ref{eq:Tpcontainment}))
there is an associated subgroup
$$
H_{\eta,s} < C_s[tor^p].
$$
The $\pi_1(S,s)$-invariance of $[\eta]_{K^p_0}$ implies that $H_{\eta,s}$ 
extends to
a local system over $S$, giving a subgroup
$$ H_{\eta} < C, $$
and a corresponding isogeny
$$ \phi_\eta: C \rightarrow C/H_\eta =: C_\eta. $$
Extending this to $\eta$ not satisfying (\ref{eq:Tpcontainment}) associates
to a $\pi_1(S,s)$-invariant $K^p_0$-level structure $[\eta]_{K^p_0}$ of
$C_s$ an quasi-isogeny $\phi_\eta: C \rightarrow C_\eta$ 
of elliptic schemes over $S$. 

Associated to a compact open subgroup 
$$ K^p \subset GL_2(\AF^{p,\infty}) $$
is a Deligne-Mumford stack $\mc{M}(K^p)$ over $\ZZ_{(p)}$ 
of \emph{elliptic curves with
$K^p$-level structure}.  For a connected scheme $S$ over $\ZZ_{(p)}$ 
with a specified
geometric point $s$, the 
$S$-points of $\mc{M}(K^p)$ are the
groupoids whose objects are tuples
$$ (C, [\eta]_{K^p}) $$
where:
\begin{align*}
C & = \text{elliptic scheme over $S$,} \\
[\eta]_{K^p} & = \text{$\pi_1(S,s)$ invariant $K^p$-level structure on
$C_s$.}
\end{align*}
The morphisms of the groupoid of $S$-points of $\mc{M}(K^p)$ 
$$ \alpha: (C, [\eta]_{K^p}) \rightarrow (C', [\eta']_{K^p}) $$
are the
prime-to-$p$ quasi-isogenies
$$ \alpha: C \rightarrow C' $$
for which
$$ [\alpha_* \circ \eta]_{K^p} = [\eta']_{K^p}. $$

\begin{rmk}\label{rmk:ellipticrmks}
$\quad$
\begin{enumerate}
\item If the compact open subgroup is given by
$$ K^p_0 := GL_2(\widehat{\ZZ}^p) \subset GL_2(\AF^{p,\infty}) $$
then there is an
isomorphism
\begin{align*}
\mc{M}(K^p_0) & \xrightarrow{\cong} \mc{M}_{ell} \otimes \ZZ_{(p)} \\
(C, [\eta]_{K^p_0}) & \mapsto C_\eta
\end{align*}
where $\mc{M}_{ell}/\ZZ$ is the usual (uncompactified) 
moduli stack of elliptic curves.

\item If the compact open subgroup is given by 
$$ K^p_0(\ell) := GL_2(\widehat{\ZZ}^{p,\ell})K_0(\ell) \subset
GL_2(\AF^{p,\infty}), $$
where $K_0(\ell) \subset GL_2(\ZZ_\ell)$ is the subgroup of matrices given
by
$$ K_0(\ell) = \left\{ A \in GL_2(\ZZ_\ell) \: : \: 
A \equiv
\begin{bmatrix}
* & * \\
0 & *
\end{bmatrix}
\mod \ell
\right\},
$$
then there is an isomorphism
\begin{align*}
\mc{M}(K^p_0(\ell)) & \xrightarrow{\cong} 
\mc{M}(\Gamma_0(\ell)) \otimes_{\ZZ[1/\ell]} \ZZ_{(p)}, \\
(C, [\eta]_{K^p_0(\ell)}) & \mapsto (C_\eta, N_\eta)
\end{align*}
where $\mc{M}(\Gamma_0(\ell))$ is the moduli stack of elliptic curves with a
$\Gamma_0(\ell)$-structure, and $N_\eta$ is the $\Gamma_0(\ell)$-structure
(subgroup of order $\ell$) of $C$ associated to the image of the composite
$$ \ell^{-1} \ZZ_\ell \oplus \ZZ_\ell \rightarrow V^p \xrightarrow{\eta} 
V^p(C) \xrightarrow{(\phi_\eta)_*} V^p(C_\eta).
$$

\item If $K^p_1 < K^p_2$ is a pair of compact open subgroups, then there is
an induced \'etale cover of moduli stacks:
\begin{align*}
\mc{M}(K^p_1) & \rightarrow \mc{M}(K^p_2), \\
(C, [\eta]_{K^p_1}) & \mapsto (C, [\eta]_{K^p_2}).
\end{align*}
If $K^p_1$ is normal in $K^p_2$, the cover is a torsor for $K^p_2/K^p_1$.

\item An element $g \in GL_2(\AF^{p,\infty})$ gives rise to an isomorphism 
of stacks
\begin{align*}
g^* : \mc{M}(gK^pg^{-1}) & \rightarrow \mc{M}(K^p), \\
(C, [\eta]_{gK^pg^{-1}}) & \mapsto (C, [\eta \circ g]_{K^p}). 
\end{align*}
Clearly we have
\begin{equation}\label{eq:Kinvariance}
g^* = (gk)^* \quad \text{for $k \in K^p$.}
\end{equation}
\end{enumerate}

\end{rmk}

The moduli of $p$-divisible groups corresponding to the moduli space
$\mc{M}(K^p)$ satisfies Lurie's generalization of the Hopkins-Miller
theorem \cite[Sec.~8.1]{BehrensLawson}, and hence the $p$-completion
$\mc{M}(K^p)^\wedge_p$ carries a presheaf of $p$-complete 
$E_\infty$-ring spectra
$ \mc{E}_{K^p} $
on the site $(\mc{M}(K^p)^\wedge_p)_{et}$
such that:
\begin{enumerate}
\item The presheaf $\mc{E}_{K^p}$ satisfies homotopy hyperdescent (i.e. it
is Jardine fibrant).
\item For an affine \'etale open
$$ \mr{Spf}(R) \xrightarrow{(C, [\eta]_{K^p})} \mc{M}(K^p)^\wedge_p $$
the corresponding spectrum of sections 
$$ E = \mc{E}_{K^p}(\mr{Spf}(R)) $$ 
is a weakly
even periodic elliptic spectrum associated to the elliptic curve $C$ (i.e.
$\pi_0(E) = R$, and the formal group associated to $E$ is isomorphic to
the formal group of $C$).
\end{enumerate}
Define $\TMF(K^p)$ to be the global sections
$$ \TMF(K^p) := \mc{E}_{K^p}(\mc{M}(K^p)^\wedge_p). $$
In particular, we have
\begin{align*}
\TMF(K^p_0) & = \TMF_{p}, \\
\TMF(K^p_0(\ell)) & = \TMF_0(\ell)_{p}.
\end{align*}

By the functoriality of Lurie's theorem \cite[8.1.4]{BehrensLawson}, the action of
$GL_2(\AF^{p,\infty})$ described in Remark~\ref{rmk:ellipticrmks} gives
rise to maps of $E_\infty$-ring spectra
\begin{equation}\label{eq:GLops}
g_*: \TMF(K^p) \rightarrow \TMF(gK^pg^{-1}).
\end{equation}

\section{The spectra $Q(S)$}\label{sec:Q(S)}

The collection of compact open subgroups $K^p$ of $GL_2(\AF^{p,\infty})$
under inclusion forms a filtered category, and we may take the colimit
\begin{equation}\label{eq:Vdef}
\mc{V} := \colim_{K^p} \TMF(K^p).
\end{equation}
The action of $GL_2(\AF^{p,\infty})$ described above gives $\mc{V}$ the
structure of a \emph{smooth $GL_2(\AF^{p,\infty})$-spectrum}
\cite[10.3]{BehrensLawson}.  We may recover each of the spectra $\TMF(K^p)$ from
$\mc{V}$ by taking homotopy fixed points \cite[10.6.5]{BehrensLawson}:
$$ \TMF(K^p) \simeq \mc{V}^{hK^p}. $$

For a set of primes $S$ not containing $p$, we have an open subgroup
$$ (K^{p,S}_0)_+ := GL_2(\AF_{S})K^{p,S}_0 \subset
GL_2(\AF^{p,\infty}) $$
where $\AF_S = \prod'_{\ell \in S} \QQ_\ell$ is the ring of $S$-adeles and
$$ K_0^{p,S} = \prod_{\ell \not\in \{ p\} \cup S} GL(\ZZ_\ell). $$
We define a spectrum
$$ Q(S) = \mc{V}^{h(K^{p,S}_0)_+}. $$

The $K(2)$-localization of the spectrum $Q(S)$ 
is closely related to the $K(2)$-local sphere, as
we now explain.
Let $C_0$ be a fixed supersingular curve over $\br{\FF}_p$ 
(any two are isogenous).  Assume (for convenience) that $C_0$ is 
defined over $\FF_p$  (such a curve exists for every prime $p$
\cite{Waterhouse}).  The quasi-endomorphism ring
$$ D := \End^0(C_0) $$
is a quaternion algebra over $\QQ$ ramified at $p$ and $\infty$.  The
subring of actual endomorphisms
$$ \mc{O}_D := \End(C_0) \subset D $$
is a maximal order.  For our set of primes $S$, define a group
$$ \Gamma_S := (\mc{O}_D[S^{-1}])^\times. $$ 
The group $\Gamma_S$ is the group of quasi-isogenies of $C_0$ whose degree
lies in 
$$ \ZZ[S^{-1}]^\times \subset \QQ^\times. $$
The group $\Gamma_S$ embeds in the (profinite) 
Morava stabilizer group through its
action on the height $2$ formal group of $C_0$:
$$ \Gamma_S \hookrightarrow \Aut(\widehat{C}_0) \cong \MS_2. $$

\begin{thm}[Behrens-Lawson \cite{BehrensLawsondense}]\label{thm:dense}
If $p$ is odd, and $S$ contains a generator of $\ZZ_p^\times$, then the
subgroup
$$ \Gamma_S \hookrightarrow \MS_2 $$
is dense.
\end{thm}

The universal deformation $\td{C}_0$ 
of the supersingular curve $C_0$, by Serre-Tate theory, 
gives Morava $E$-theory $E_2$ the structure of an
elliptic spectrum, where
$$ \pi_*E_2 = W(\br{\FF}_p)[[u_1]][u^{\pm 1}]. $$
Since $C_0$ is assumed to admit a definition over $\FF_p$, there is an
action of the Galois group $\Gal(\FF_p)$ on the spectrum $E_2$.
Picking a fixed rational uniformization
$$ \eta_0 : V^p \xrightarrow{\cong} V^p(C_0) $$
gives, for every $K^p$, a canonical map of $E_\infty$-ring spectra
$$ \TMF(K^p) \xrightarrow{(\td{C}_0, [\eta_0]_{K^p})} E_2 $$
classifying the pair $(\td{C}_0, [\eta_0]_{K^p})$, thus a map
\begin{equation}\label{eq:Vmap}
\mc{V} \xrightarrow{(\td{C}_0, \eta_0)} E_2.
\end{equation}
Using the Tate embedding
$$ \End^0(C_0) \hookrightarrow \End(V^p(C_0)), $$
the rational uniformization $\eta_0$ induces an inclusion
\begin{align*}
\gamma: \Gamma_S & \hookrightarrow (K^{p,S}_0)_+ \subseteq
GL_2(\AF^{p,\infty}), \\
\alpha & \mapsto \eta_0^{-1} \alpha \eta_0. 
\end{align*}

\begin{lem}\label{lem:equivariance}
$\quad$
\begin{enumerate}
\item
For $\alpha \in \Gamma_S$, the following diagram commutes
$$
\xymatrix@C+2em{
\TMF(K^p) \ar[r]^{(\td{C}_0, \eta_0)} \ar[d]_{\gamma(\alpha)_*} 
& E_2 \ar[d]^{\alpha_*} 
\\
\TMF(\gamma(\alpha)K^p \gamma(\alpha)^{-1}) \ar[r]_-{(\td{C}, \eta_0)}
& E_2
}
$$
where $\gamma(\alpha)_*$ is the morphism induced by the action of
$GL_2(\AF^{p,\infty})$ on $\TMF$, and $\alpha_*$ is the morphism induced by
the action of the Morava stabilizer group on $E_2$ through the inclusion
$\Gamma_S \subset \MS_2$.

\item
The map
$$ \TMF(K^p) \xrightarrow{(\td{C}_0, \eta_0)} E_2 $$
is invariant under the action of $\Gal(\FF_p)$ on $E_2$.
\end{enumerate}
\end{lem}

\begin{proof}
Let $\mr{Def}(C_0)$ denote the formal moduli of deformations of $C_0$.  
For a complete local ring $(R, \mathfrak{m})$, 
the $R$-points of $\mr{Def}(C_0)$ is the
category of tuples 
$$ (C, \iota, \beta), $$
where 
\begin{align*}
C & = \text{elliptic curve over $R$, with reduction $\br{C}$ mod 
$\mathfrak{m}$,} \\
\iota & : \br{\FF}_p \rightarrow R/\mathfrak{m}, \\ 
\beta & : \iota^*C_0 \xrightarrow{\cong} \br{C}.
\end{align*}
The element $\alpha \in \Gamma_S$ acts on the $R$-points of $\mr{Def}(C_0)$ by
$$ \alpha^*: (C, \iota, \beta) \mapsto (C, \iota, \beta \circ \alpha). $$
By
Serre-Tate theory, this space is equivalent to the formal moduli
$\mr{Def}(\widehat{C}_0)$ of
deformations of the associated height $2$ formal group $\widehat{C}_0$, and
the action of $\Gamma_S$ on $\mr{Def}(C_0)$ is compatible with the action
of $\MS_2$ on $\mr{Def}(\widehat{C}_0)$.
Part (1) of the lemma follows from the commutativity of the following
diagram, which is easily checked on $R$-points.
$$
\xymatrix@C+5em{
\mc{M}(K^p) & 
\mr{Def}(C_0) \ar[l]_{(C_{\mr{univ}}, [(\beta_{\mr{univ}})_* \eta_0])} 
\\
\mc{M}(\gamma(\alpha)K^p\gamma(\alpha)^{-1}) \ar[u]^{\gamma(\alpha)^*} 
& \mr{Def}(C_0) \ar[u]_{\alpha^*} \ar[l]^-{(C_{\mr{univ}}, 
[(\beta_{\mr{univ}})_* \eta_0 ])}
}
$$
Part (2) is checked in a similar manner.
\end{proof}

Lemma~\ref{lem:equivariance} implies that the morphism (\ref{eq:Vmap})
descends to give a morphism
\begin{equation}\label{eq:map}
Q(S) = \mc{V}^{h(K^{p,S})_+} \rightarrow \left( E_2^{h\Gamma_S}
\right)^{h\Fr_p} =: E(\Gamma_S). 
\end{equation}
Here, if $X$ is a spectrum with an action of the Frobenius $\Fr_p \in
\Gal(\FF_p)$, the spectrum $X^{h\Fr_p}$ is defined to be the homotopy fiber
$$ X^{h\Fr_p} \rightarrow X \xrightarrow{\Fr_p - 1} X. $$

The following theorem is proved in \cite{Behrens} in the case where $S$
consists of one prime.  The proof of the more general case is identical to
the proof of Corollary~14.5.6 of \cite{BehrensLawson}.

\begin{thm}\label{thm:K2Ql}
The map (\ref{eq:map}) induces an equivalence
$$ Q(S)_{K(2)} \rightarrow E(\Gamma_S). $$
\end{thm}

\section{The building resolution}\label{sec:building}

If $S = \{\ell\}$ is a set containing one prime, the spectrum $Q(\ell)$
defined in Section~\ref{sec:Q(S)} is equivalent to the spectrum 
constructed in \cite{Behrens}.
In \cite{Behrens}, the spectrum $Q(\ell)$ was defined to be the
totalization of a certain semi-cosimplicial $E_\infty$-ring spectrum.  This
description is recovered as follows.

The group $GL_2(\QQ_\ell)$ acts on its building
$\mc{B} = \mc{B}(GL_2(\QQ_\ell))$ with compact open stabilizers.  
Explicitly, the building $\mc{B}$ is equivariantly
homeomorphic to the geometric realization of a semi-simplicial
$GL_2(\QQ_\ell)$-set $\mc{B}_\bullet$ of the form
\begin{equation}\label{eq:BGLQl}
\mc{B}_\bullet = 
\left(
GL_2(\QQ_\ell)/GL_2(\ZZ_\ell)
\begin{array}{c}
\leftarrow \\
\leftarrow
\end{array}
\begin{array}{c}
GL_2(\QQ_\ell)/K_0(\ell) \\
\times \\
GL_2(\QQ_\ell)/GL_2(\ZZ_\ell)
\end{array}
\begin{array}{c}
\leftarrow \\
\leftarrow \\
\leftarrow
\end{array}
GL_2(\QQ_\ell)/K_0(\ell).
\right)
\end{equation}
The action of $GL_2(\QQ_\ell)$ on the building $\mc{B}$ extends to an
action of $(K^{p,\ell}_0)_+$, simply by letting the local factors away
from $\ell$ act trivially.  Regarded as a semi-simplicial
$(K^{p,\ell}_0)_+$-set, we have 
\begin{equation}\label{eq:BGLAF}
\mc{B}_\bullet = \left(
(K^{p,\ell}_0)_+/K^p_0
\begin{array}{c}
\leftarrow \\
\leftarrow
\end{array}
\begin{array}{c}
(K^{p,\ell}_0)_+/K^p_0(\ell) \\
\times \\
(K^{p,\ell}_0)_+/K^p_0
\end{array}
\begin{array}{c}
\leftarrow \\
\leftarrow \\
\leftarrow
\end{array}
(K^{p,\ell}_0)_+/K^p_0(\ell)
\right).
\end{equation}

The canonical $(K^{p,\ell}_0)_+$-equivariant
morphism
\begin{equation}\label{eq:cmorphism}
\mc{V} \xrightarrow{c} \ul{\Map}(\mc{B}, 
\mc{V})^{sm} 
\end{equation}
(given by the inclusion of the constant functions) is an equivalence
\cite[Lem.~13.2.3]{BehrensLawson}.  Here, $\ul{\Map}(-,-)^{sm}$ is defined
to be the colimit of the $U$-fixed point spectra, as $U$ ranges over the
open subgroups of $GL_2(\AF^{p,\infty})$.  The argument in
\cite{BehrensLawson} relies on the fact that the building $\mc{B}$ is not
only contractible, but possesses a canonical contracting homotopy with
excellent equivariance properties.

The semi-simplicial decomposition of $\mc{B}$ induces an equivariant equivalence
$$ \ul{\Map}(\mc{B}, \mc{V})^{sm} \simeq \Tot \ul{\Map}(\mc{B}_\bullet, \mc{V})^{sm}
$$
and therefore an equivalence on fixed point spectra:
\begin{align*}
Q(\ell) & = \mc{V}^{h(K^{p,\ell}_0)_+} \\
& \simeq \Tot \left( \ul{\Map}(\mc{B}_\bullet, \mc{V})^{sm}
\right)^{h(K^{p,\ell}_0)_+}.
\end{align*}
Using Shapiro's lemma in the context of smooth
$(K^{p,\ell}_0)_+$-spectra gives an equivalence
$$ (\ul{\Map}((K^{p,\ell}_0)_+/K^p, \mc{V})^{sm})^{h(K^{p,\ell}_0)_+}
\simeq \mc{V}^{hK^p}.
$$
Since we have
\begin{gather*}
\mc{V}^{hK^p_0} \simeq \TMF(K^p_0) = \TMF_p \\
\mc{V}^{hK^{p}_0(\ell)} \simeq \TMF(K^p_0(\ell)) = \TMF_0(\ell)_p
\end{gather*}
Thus there is an induced semi-cosimplicial decomposition
\begin{equation}\label{eq:cosimplicial}
Q(\ell) \simeq \Tot Q(\ell)^\bullet 
\end{equation}
where
\begin{equation}\label{eq:cosimplicialQ}
Q(\ell)^\bullet =
\left( \TMF_p
\begin{array}{c}
\rightarrow \\
\rightarrow
\end{array}
\begin{array}{c}
\TMF_0(\ell)_p \\
\times \\
\TMF_p
\end{array}
\begin{array}{c}
\rightarrow \\
\rightarrow \\
\rightarrow
\end{array}
TMF_0(\ell)_p
\right).
\end{equation}

For $p \ge 5$, the homotopy groups of $\TMF_p$ and $\TMF_0(\ell)_p$ are
concentrated in even degrees, and there are isomorphisms
\begin{align*} 
\pi_{2k}\TMF_p & \cong (M_k)^0_{\ZZ_p}, \\
\pi_{2k}\TMF_0(\ell)_p & \cong M_k(\Gamma_0(\ell))^0_{\ZZ_p}.
\end{align*}

Applying homotopy to the semi-cosimplicial spectrum $Q(\ell)^\bullet$
(\ref{eq:cosimplicialQ}) gives
a semi-cosimplicial abelian group
\begin{equation}\label{eq:cosimplicialgp}
C(\ell)_{2k}^\bullet := \left( (M_k)_{\ZZ_p}^0
\begin{array}{c}
\rightarrow \\
\rightarrow
\end{array}
\begin{array}{c}
M_k(\Gamma_0(\ell))_{\ZZ_p}^0 \\
\times \\
(M_k)_{\ZZ_p}^0
\end{array}
\begin{array}{c}
\rightarrow \\
\rightarrow \\
\rightarrow
\end{array}
M_k(\Gamma_0(\ell))_{\ZZ_p}^0
\right).
\end{equation}

The Bousfield-Kan spectral sequence for $Q(\ell)^\bullet$ takes the form
\begin{equation}\label{eq:BKSSQl}
E_1^{s,t} = C(\ell)_t^s \Rightarrow \pi_{t-s} Q(\ell).
\end{equation}

\begin{prop}\label{prop:collapse}
For $p \ge 5$, the spectral sequence (\ref{eq:BKSSQl}) collapses at $E_2$
to give an isomorphism
$$ \pi_n Q(\ell) \cong H^0(C(\ell)^\bullet_n) \oplus
H^{1}(C(\ell)^\bullet_{n+1}) \oplus H^2(C(\ell)^\bullet_{n+2}). $$
\end{prop}

\begin{proof}
The rings of modular forms $M_*$ and $M_*(\Gamma_0(\ell))$ are concentrated
in even weights.  This easily follows in the case of $\Gamma_0(\ell)$ from
the fact that the inversion $[-1]: C \rightarrow C$ gives an automorphism
of any $\Gamma_0(\ell)$-structure.  Thus there is no room for
differentials, or hidden extensions, in the spectral sequence
(\ref{eq:BKSSQl}).
\end{proof}

In fact, since we have argued that $H^s(C(\ell)^\bullet_t)$ is non-zero
unless $t \equiv 0 \mod 4$, we have

\begin{cor}\label{cor:collapseisomorphism1.5}
For $p \ge 5$, there are isomorphisms
$$
\pi_t Q(\ell) \cong H^s(C(\ell)^\bullet)_{t+s}
$$
where $t = 4k-s$ and $0 \le s < 4$.
\end{cor}

\section{Effect of coface maps on modular forms}\label{sec:coface}

Suppose that $p \ge 5$.
In this section we will deduce the effect of the two initial cosimplicial coface maps of
$C(\ell)^\bullet_{2k}$ on
the level of $q$-expansions.
To aid in this, we recall from \cite{Behrens} that the semi-cosimplicial
resolution of $Q(\ell)$ may be constructed by applying the
Goerss-Hopkins-Miller presheaf to a semi-simplicial object in the site
$(\mc{M}_{ell})_{et}$:
\begin{equation}\label{eq:simpstack}
\mc{M}_\bullet := \left( (\mc{M}_{ell})_p
\begin{array}{c}
\leftarrow \\
\leftarrow
\end{array}
\begin{array}{c}
\mc{M}(\Gamma_0(\ell))_p \\
\amalg \\
(\mc{M}_{ell})_p
\end{array}
\begin{array}{c}
\leftarrow \\
\leftarrow \\
\leftarrow
\end{array}
\mc{M}(\Gamma_0(\ell))_p
\right).
\end{equation}
The coface maps $d_i: \mc{M}_1 \rightarrow \mc{M}_0$ are given on
$R$-points by
\begin{align*}
d_i: (\mc{M}_{ell})_p & \rightarrow (\mc{M}_{ell})_p \\
d_0: C & \mapsto C/C[\ell] \\
d_1: C & \mapsto C \\
d_i: (\mc{M}(\Gamma_0(\ell)))_p & \rightarrow (\mc{M}_{ell})_p \\
d_0: (C,H) & \mapsto C/H \\
d_1: (C,H) & \mapsto C \\
\end{align*}

\begin{prop}\label{prop:diffq}
Consider the morphisms
$$ d_0, d_1 : (M_k)^0_{\ZZ_p} \rightarrow M_k(\Gamma_0(\ell))^0_{\ZZ_p} \times
(M_k)^0_{\ZZ_p} $$
induced by the initial coface maps of the cosimplicial abelian group
$C(\ell)^\bullet_{2k}$.
On the level of $q$-expansions, the maps are given by
\begin{align*}
d_0(f(q)) & := (\ell^kf(q^\ell), \ell^kf(q)), \\
d_1(f(q)) & := (f(q), f(q)).
\end{align*}
\end{prop}

\begin{proof}
It is clear from the description of the map $d_1$ that its effect on
$q$-expansions is as given. 
Choosing an embedding $\ZZ_p \hookrightarrow \CC$, by the $q$-expansion
principle, it suffices to verify
these identities hold when we base-change to $\CC$ and consider the Tate
curve:
$$ C_q := \CC^\times/q^\ZZ. $$
The group of $\ell$th roots of unity $\mu_\ell \subset \CC^\times$ induces  
a $\Gamma_0(\ell)$-structure on the Tate curve $C_q$:
$$ \mu_\ell \subset \CC^\times/q^\ZZ = C_q. $$
This level structure is the kernel of the isogeny
\begin{align*}
\phi_{\mu_{\ell}}: C_q = \CC^\times/q^\ZZ & \rightarrow \CC^\times/q^{\ell
\ZZ} = C_{q^\ell}, \\
z & \mapsto z^\ell.
\end{align*}
The invariant differential $dz/z$ on $C_q$ transforms under this isogeny by
$$ \phi^* (dz/z) = \ell dz/z. $$
It follows that $d_0$ on the $M_k(\Gamma_0(\ell))^0_{\ZZ_p}$-component is 
given by
\begin{align*}
d_0: (M_k)^0_{\ZZ_p} & \rightarrow M_k(\Gamma_0(\ell))^0_{\ZZ_p}, \\
f(q) & \mapsto \ell^k f(q^{\ell}),
\end{align*}
as desired.  The $\ell$th power map
\begin{align*}
[\ell]: C_q & \rightarrow C_q \\
z & \mapsto z^\ell 
\end{align*}
transforms the invariant differential by
$$ [\ell]^* (dz/z) = \ell dz/z. $$
If follows that the $(M_k)^0_{\ZZ_p}$ component of $d_0$ is given by
\begin{align*}
d_0: (M_k)^0_{\ZZ_p} & \rightarrow (M_k)^0_{\ZZ_p}, \\
f(q) & \mapsto \ell^k f(q).
\end{align*}
\end{proof}

\section{The chromatic spectral sequence for $Q(\ell)$}\label{sec:cssQl}

The following lemma implies that the chromatic resolution of $Q(\ell)$ is
finite.

\begin{lem}\label{lem:E2local}
The spectrum $Q(\ell)$ is $E(2)$-local.
\end{lem}

\begin{proof}
The spectra $\TMF_p$ and $\TMF_0(\ell)_p$ are $E(2)$-local.  By
(\ref{eq:cosimplicial}), the spectrum $Q(\ell)$ is $E(2)$-local. 
\end{proof}

We deduce that the chromatic resolution for $Q(\ell)$ takes the following
form.
$$
\xymatrix{
M_0Q(\ell) \ar@{=}[d] & 
M_1Q(\ell) \ar[d] &
M_2Q(\ell) \ar[d] &
\\
Q(\ell)_{E(0)} &
Q(\ell)_{E(1)} \ar[l] &
Q(\ell) \ar[l] &
}
$$
Applying homotopy, we get a three line spectral sequence
\begin{equation}
E_1^{n,k} = 
\left\{
\begin{array}{ll}
\pi_k M_n Q(\ell), & n \le 2 \\
0, & n > 2
\end{array}
\right\}
\Rightarrow \pi_{k-n} Q(\ell).
\end{equation}
Assuming that $p \ge 5$, applying $M_n$ to the cosimplicial resolution 
(\ref{eq:cosimplicial}), we
get spectral sequences
\begin{gather}
E_2^{s,t} = H^s(C(\ell)^\bullet[p^{-1}])_t \Rightarrow \pi_{t-s} M_0Q(\ell)
\label{eq:M0Qlss}
\\
E_2^{s,t} = H^s(C(\ell)^\bullet/p^\infty[v_1^{-1}])_t \Rightarrow \pi_{t-s}
M_1Q(\ell) \label{eq:M1Qlss}
\\
E_2^{s,t} = H^s(C(\ell)^\bullet/(p^\infty, v_1^\infty))_t \Rightarrow
\pi_{t-s} M_2Q(\ell) \label{eq:M2Qlss}
\end{gather}
Here, the $E_2$-terms are the cohomology of the cosimplicial abelian group
obtained from applying the functor $\pi_*(M_n -)$ to
(\ref{eq:cosimplicial}).  The values of the resulting cosimplicial abelian
group are given by the following lemma.

\begin{lem}\label{lem:MnTMF}
Let $p \ge 5$ and $(N,p) = 1$.  Then
\begin{align*}
\pi_{2*}M_0\TMF_0(N)_p & = M_*(\Gamma_0(N))_{\QQ_p}^0,
\\ 
\pi_{2*}M_1\TMF_0(N)_p & = M_*(\Gamma_0(N))^0_{\ZZ_p}/p^\infty
[E_{p-1}^{-1}],
\\
\pi_{2*}M_2\TMF_0(N)_p & = M_*(\Gamma_0(N))^0_{\ZZ_p}/(p^\infty,
E_{p-1}^{\infty}). 
\end{align*}
\end{lem}

\begin{proof}
This is a direct application of (\ref{eq:Mncolimit}). 
Here, $E_{p-1}$ is the $(p-1)$st Eisenstein series, which reduces to the
Hasse invariant $v_1$ mod $p$ \cite[Sec.~2.1]{Katz}.
\end{proof}

In particular, since $C(\ell)^\bullet_t$ 
is non-zero only when $t \equiv
0 \mod 4$, the same argument proving
Corollary~\ref{cor:collapseisomorphism1.5} gives:

\begin{cor}\label{cor:collapseisomorphism2}
For $p \ge 5$, there are isomorphisms
\begin{align*}
\pi_t M_0Q(\ell) & \cong H^s(C(\ell)^\bullet[p^{-1}])_{t+s} \\
\pi_t M_1Q(\ell) & \cong H^s(C(\ell)^\bullet/(p^\infty)[v_1^{-1}])_{t+s} \\
\pi_t M_2Q(\ell) & \cong H^s(C(\ell)^\bullet/(p^\infty, v_1^\infty))_{t+s}
\end{align*}
where $t = 4k-s$ and $0 \le s < 4$.
\end{cor}

We end this section by relating $M_2Q(\ell)$ to the subgroup $\Gamma_\ell
\subset \MS_n$.  By Theorem~\ref{thm:K2Ql} there are equivalences
$$ M_2 Q(\ell) \simeq M_2 (Q(\ell)_{K(2)}) \simeq M_2
((E_2^{h\Gamma_\ell})^{h\Fr_p}). $$

We recall the following result from \cite{Behrens}.  

\begin{prop}\label{prop:actionbldg}
The group $\Gamma_\ell$ acts on the building $\mc{B}$ for $GL_2(\QQ_\ell)$
with finite stabilizers, given by groups of automorphisms of 
supersingular curves.
\end{prop}

We deduce the following.

\begin{lem}
There is an equivalence 
$$ 
M_2((E_2^{h\Gamma_\ell})^{h\Fr_p}) \simeq 
((M_2E_2)^{h\Gamma_\ell})^{h\Fr_p}.
$$
\end{lem}

\begin{proof}
Since the spectra $M(I)^0$ are finite, we have
$$
M(I)^0 \wedge ((E_2^{h\Gamma_\ell})^{h\Fr_p}) \simeq 
((M(I)^0 \wedge E_2)^{h\Gamma_\ell})^{h\Fr_p}.
$$
The result would follow from (\ref{eq:Mncolimit}) if we could commute the
homotopy colimit over $I$ with the homotopy fixed points with respect to
$\Gamma_\ell$.  However, by Proposition~\ref{prop:actionbldg}, 
the group $\Gamma_\ell$ acts on the building
$\mc{B}$ for $GL_2(\QQ_\ell)$ with finite stabilizers.
Since $\mc{B}$ is contractible and finite dimensional, we conclude that 
the group $\Gamma_\ell$ has finite virtual cohomological dimension.
\end{proof}

We conclude that there is an equivalence
$$ M_2 Q(\ell) \simeq ((M_2 E_2)^{h\Gamma_\ell})^{h\Fr_p} $$
and a homotopy fixed point spectral sequence
\begin{equation}\label{eq:Gammahfpss} 
E_2^{s,t} = H^s(\Gamma_\ell, \pi_t M_2 E_2)^{\Gal(\FF_p)} \Rightarrow 
\pi_{t-s} M_2Q(\ell).
\end{equation}

\begin{lem}\label{lem:collapseisomorphism3}
For $p \ge 5$, we have:
\begin{enumerate}
\item $H^s(\Gamma_\ell, \pi_t M_2E_2)^{\Gal(\FF_p)} = 0$ for $s > 2$. 
\item $H^s(\Gamma_\ell, \pi_t M_2E_2)^{\Gal(\FF_p)} = 0$ for $t \not\equiv
0 \mod 4$.
\item There are isomorphisms
$$ \pi_tM_2Q(\ell) \cong H^s(\Gamma_\ell, \pi_{t+s}M_2E_2)^{\Gal(\FF_p)} $$
where $t = 4k - s$ and $0 \le s < 4$.
\end{enumerate}
\end{lem}

\begin{proof}
(1) follows from Proposition~\ref{prop:actionbldg}, together with the fact
that the coefficients are $p$-local and the building is contractible and
$2$-dimensional.  (2) follows from the fact that there is a central,
Galois invariant element $[-1] \in \Gamma_\ell$ (given by inversion) which
acts on $\pi_{2i} M_2 E_2$ by $(-1)^i$.  (3) follows from (1) and (2),
using the spectral sequence (\ref{eq:Gammahfpss}).
\end{proof}

Combining Corollary~\ref{cor:collapseisomorphism2} with
Lemma~\ref{lem:collapseisomorphism3}, we get

\begin{cor}\label{cor:collapseisomorphism4}
For $p \ge 5$, there are isomorphisms
$$ H^s(C(\ell)^\bullet/(p^\infty, v_1^\infty))_t \cong H^s(\Gamma_\ell,
\pi_t M_2 E_2)^{\Gal(\FF_p)}. $$
\end{cor}

\begin{rmk}
One could give a purely algebraic proof of
Corollary~\ref{cor:collapseisomorphism4} which makes no reference to
topology.  In the context of the exposition of this paper it happens to be 
quicker
(but arguably less natural) to use topological constructions.
\end{rmk}

\section{$M_0Q(\ell)$}\label{sec:M0Ql}

Let
$p \ge 5$ and  
$\ell$ be a topological generator of $\ZZ_p^\times$.  
In this section we will concern ourselves with locating the non-trivial
homotopy of $M_0Q(\ell)$.

\begin{prop}\label{prop:torsion}
The groups
$$ H^s(C^\bullet(\ell))_{2t} $$
consist entirely if $p^j$-torsion if
$$ t \equiv 0 \mod (p-1)p^{j-1}, $$
and are zero if $t \not\equiv 0 \mod (p-1)$. 
\end{prop}

\begin{proof}
Consider the central element
$$ [\ell] := 
\begin{bmatrix}
\ell & 0 \\
0 & \ell
\end{bmatrix} \in GL_2(\QQ_\ell).
$$
Let $\mc{V}$ be the smooth $GL_2(\AF^{p,\infty})$-spectrum of
(\ref{eq:Vdef}).  We assume that $\mc{V}$ is fibrant as a smooth
$GL_2(\AF^{p,\infty})$-spectrum, so that homotopy fixed points are
equivalent to point-set level fixed points
$$ \mc{V}^{hU} \simeq \mc{V}^U $$
for $U$ an open subgroup of $GL_2(\AF^{p,\infty})$
\cite[Cor.10.5.5]{BehrensLawson}.
Because $[\ell]$ is central, 
the action of $[\ell]$ on $\mc{V}$ is $GL_2(\AF^{p,\infty})$-equivariant.
Because $[\ell]$ is contained in the subgroup $(K^{p,\ell}_0)_+$, it acts
as the identity on 
$$ Q(\ell) \simeq \mc{V}^{(K^{p,\ell}_0)_+}. $$
However, the morphism $c$ of (\ref{eq:cmorphism}) is compatible with the
action of $[\ell]$, where we let $[\ell]$ act on $\ul{\Map}(\mc{B},
\mc{V})^{sm}$ through its action on the target $\mc{V}$.  We deduce that
the endomorphism $[\ell]$ acts on the cosimplicial object
$Q(\ell)^\bullet$, where the action is given level-wise on each factor 
by the endomorphism
$$ [\ell]: \TMF(K^p) \rightarrow \TMF(K^p) $$
(where $K^p$ is either $K^p_0$ or $K^p_0(\ell)$).
The endomorphism $[\ell]$ is the induced action of $[\ell]$ on the fixed
point spectrum
$$ \TMF(K^p) \simeq \mc{V}^{hK^p}. $$
The action of $[\ell]$ on the homotopy groups of 
$\TMF(K^p)$ is given 
by 
\begin{align*}
[\ell]: \pi_{2k}\TMF(\Gamma_0(N)) & \rightarrow \pi_{2k}\TMF(\Gamma_0(N)), \\
f & \mapsto \ell^k f.
\end{align*}
This is easily deduced from the fact that the induced quasi-isogeny
(\ref{eq:quasiisogeny})
$$ C_\eta \rightarrow C_{\eta \circ [\ell]} $$
is isomorphic to the $\ell$th power map of elliptic curves.
It follows that
$$ [\ell]: H^s(C(\ell)^\bullet_{2k}) \rightarrow
H^s(C(\ell)^\bullet_{2k}) $$
acts by multiplication by $\ell^k$.  However, since we have shown that
$[\ell]$ acts by the identity on $\pi_* Q(\ell)$,
Proposition~\ref{prop:collapse} implies that $[\ell]$ acts by the identity
on $H^s(C(\ell)^\bullet_{2k})$.  We deduce that multiplication by
$\ell^k-1$ is the zero homomorphism on $H^s(C(\ell)^\bullet_{2k})$.  Since
$\ell$ was assumed to be a topological generator of $\ZZ_p^\times$, the
proposition follows.
\end{proof}

We immediately deduce:

\begin{cor}
We have 
$$ H^s(C^\bullet(\ell)[p^{-1}])_t = 0 $$
for $t \ne 0$.
\end{cor}

We can be more specific in the case of $s = 0$.

\begin{lem}
We have 
$$ H^0(C^\bullet(\ell)[p^{-1}])_0 = \QQ_p. $$
\end{lem}

\begin{proof}
We must analyze the kernel of the cosimplicial differential
$$ d_0 - d_1 : (M_0)^0_{\QQ_p} \rightarrow
M_0(\Gamma_0(\ell))_{\QQ_p}^0 \oplus (M_0)_{\QQ_p}^0.
$$
We claim that is is given by the subspace generated by $1 \in
(M_0)_{\QQ_p}$.  Indeed, suppose that $f \in (M_0)_{\QQ_p}^0$
satisfies
$$ d_0(f) - d_1(f) = 0. $$
By Proposition~\ref{prop:diffq}, it follows that
$$
f(q^\ell) - f(q) = 0.
$$
Writing $f(q) = \sum a_n q^n$, we find
$$ a_n = 
\begin{cases}
a_{n/\ell}, & n \equiv 0 \mod \ell \\
0, & n \not\equiv 0 \mod \ell.
\end{cases}
$$
It follows by induction that $f(q) = a_0$.
\end{proof}

Applying this knowledge to the spectral sequence (\ref{eq:M0Qlss}), we
deduce:

\begin{cor}\label{cor:M0Ql}
We have
$$ \pi_t(M_0Q(\ell)) = 0 $$
if $t \not\in \{ 0, -1, -2\}$, and
$$ \pi_0(M_0Q(\ell)) = \QQ_p. $$
\end{cor}

\section{$M_1Q(\ell)$: Eisenstein series and the
$\alpha$-family}\label{sec:M1Ql}

Let $p \ge 5$ and assume that $\ell$ is a topological generator of
$\ZZ_p^\times$.
In this section we will compute
$$ H^0(C^\bullet(\ell)/p^\infty), $$
the $0$th cohomology of the cochain complex associated to the cosimplicial
abelian group $C^\bullet(\ell)$ tensored with the group $\ZZ/p^\infty$.
These computations will allow us to determine part of the $1$-line of the
chromatic spectral sequence for $Q(\ell)$.

We have
$$ H^0(C^\bullet(\ell)/p^\infty) = \varinjlim_j H^0(C^\bullet(\ell)/p^j), $$
so it suffices to compute the latter.  Our explicit determination of the
first differential in $C^\bullet(\ell)$ implies that
\begin{align}
\mc{A}_{t/j} & := H^0(C^\bullet(\ell)/p^j)_{2t} \\
& = \left\{ f \in (M_t)_{\ZZ/p^j}^0 \: : \: 
\begin{array}{l}
\mr{(i)} \: (\ell^t - 1)f(q) \equiv 0 \mod p^j, \\
\mr{(ii)} \: \ell^t f(q^\ell) - f(q) \equiv 0 \mod p^j
\end{array}
\right\}
\label{eq:mcA}
\end{align}

\begin{lem}\label{lem:mcA}
A modular form 
$f \in M^0_t$ represents an element of the group $\mc{A}_{t/j}$ 
if and only if 
\begin{enumerate}
\item $p^i f \equiv 0 \mod p^j$ for $t = (p-1)p^{i-1}s$, and $(s,p) = 1$,
\item $f(q) \equiv a \mod p^j$ for $a \in \ZZ/p^j$.
\end{enumerate}
\end{lem}

\begin{proof}
Since $\ell$ was assumed to be
a topological generator of $\ZZ_p$, 
$$ \nu_p(\ell^t-1) = i $$
for $t = (p-1)p^{i-1}s$, with $(s,p) = 1$.  
Condition (i) of (\ref{eq:mcA})
states that 
$$ (\ell^t - 1)f \equiv 0 \mod p^j. $$
This proves (1).

Because $\ell^t f(q) \equiv f(q) \mod p^j$, we deduce that 
condition (ii) of (\ref{eq:mcA}) may be rewritten as
$$ f(q^\ell) \equiv f(q) \mod p^j. $$
But, writing 
$$ f(q) = \sum_{n} a_n q^n $$
for $a_n \in \ZZ/p^j$, we see that
$$
a_n \equiv
\begin{cases}
0, & n \equiv 0 \mod \ell, \\
a_{n/\ell}, & n \not\equiv 0 \mod \ell.
\end{cases}
$$
Therefore, we inductively deduce that $a_n \equiv 0 \mod p^j$ unless $n =
0$.
\end{proof}

Let $E_k \in M_{k}$ denote the weight $k$ normalized Eisenstein series (for
$k \ge 4$ even), with
$q$-expansion
\begin{equation}
E_k(q) = 1 - \frac{2k}{B_k} \sum_{i = 1}^\infty \sigma_{k-1}(i)q^i \in
\QQ[[q]],
\end{equation}
where
$$
\sigma_k(i) := \sum_{d\vert i} d^k.
$$

The following lemma follows immediately from the Clausen-von Staudt Theorem on
denominators of Bernoulli numbers.

\begin{lem}\label{lem:vonStaudt}
If $p-1$ divides
$k$, the $q$-expansion of $E_k$ is $p$-integral.
For $k \equiv 0 \mod (p-1)p^{j-1}$ we have
$$ E_k(q) \equiv 1 \mod p^{j}. $$
\end{lem}

\begin{lem}\label{lem:ek}
For each even weight $k \ge 4$ there exists a modular form
$$ e_k \in (M_k)^0_{\ZZ_{(p)}} $$
such that
\begin{enumerate}
\item if $k \equiv 0 \mod p-1$, we have 
$$ e_k = E_k, $$
\item the $q$-expansion of $e_k$ satisfies
$$ e_k(q) = 1 + \text{higher terms}, $$
\item if $k_1 \equiv k_2 \mod (p-1)p^{j-1}$, then
$$ e_{k_1}(q) \equiv e_{k_2}(q) \mod p^j. $$ 
\end{enumerate}
\end{lem}

\begin{proof}
Observe that for any even $k \ge 4$, there exist modular forms $e_k$
satisfying condition (2) (one can simply take $e_k = E_4^i E_6^j$ for
appropriate $i$ and $j$).  Fix such choices of $e_k$ for even $k$
satisfying $4 \le k < p-1$
and $k = p + 1$.  Also set $e_0 = 1$.  For even $k \ge p-1$ satisfying $k
\ne p+1$ set
$$ 
e_k = e_{k-(p-1)t} E_{(p-1)t} 
$$
for $t$ chosed such that
$$ 0 \le k - (p-1)t < p-1 \quad \text{or} \quad k - (p-1)t = p+1. $$
Then condition (1) is obviously satisfied, and condition (3) is satisfied
by Lemma~\ref{lem:vonStaudt}. 
\end{proof}

The following lemma provides a convenient basis for $p$-integral modular
forms which we shall make frequent use of.

\begin{lem}\label{lem:basis}
The forms
$$
\{ \Delta^k e_{t-12k} \: : \: k \in \ZZ, t-12k \ge 4 \: \text{and even} \},
$$
together with
$$ \Delta^k \qquad \text{if $t = 12k$}, $$
form an integral basis of $M_t^0$.
\end{lem}

\begin{proof}
Since 
$$ \Delta(q) = q + \cdots $$
we have 
$$ \Delta^k(q) e_{t-12k} = q^k + \cdots. $$
This establishes linear independence.  We may deduce that these forms span
$M_t^0$ by the explicit calculation
$$ M_* = \ZZ[E_4, E_6, \Delta^{-1}]/(\Delta = \frac{E_4^3 - E_6^2}{1728}). $$
\end{proof}

\begin{prop}\label{prop:eisenstein}
The groups $\mc{A}_{t/\infty} = \colim_j \mc{A}_{t/j}$ are given by
$$ \mc{A}_{t/\infty} = \ZZ/{p^j}\{ E_t/p^j \} $$
for $t = (p-1)p^{j-1}s$, where $(s,p) = 1$ and $t \ge 4$, and
$$ \mc{A}_{0/\infty} = \ZZ/p^\infty. $$
(Here, the element $E_t/p^j$ is
the image of the element $E_t \in \mc{A}_{t/j}$.)
\end{prop}

\begin{proof}
This follows immediately from Lemmas~\ref{lem:mcA} and \ref{lem:basis},
provided we can show that $E_t$ lies in $\mc{A}_{t/j}$.  This again follows
from criterion (2) of Lemma~\ref{lem:mcA}: by Lemma~\ref{lem:vonStaudt}
$$ E_t(q) \equiv 1 \mod p^j. $$
\end{proof}

We obtain the zero-line of spectral sequence (\ref{eq:M0Qlss}) as a corollary.

\begin{cor}\label{cor:M1QlH0}
We have
$$ H^0(C^\bullet(\ell)/p^\infty[v_1^{-1}])_{2t} \cong 
\begin{cases}
\ZZ/p^j, & t = (p-1)p^{j-1}s \: \text{and} \: (s,p) = 1, \\
0, & t \not\equiv 0 \mod (p-1).
\end{cases}
$$
\end{cor}

Combining this with Corollary~\ref{cor:collapseisomorphism2} and
Proposition~\ref{prop:torsion}, we find:

\begin{cor}\label{cor:M1Ql}
We have
$$ \pi_t M_1Q(\ell) \cong 
\begin{cases}
\ZZ/p^j, & t = 2(p-1)p^{j-1}s \: \text{and} \: (s,p) = 1, \\
0, & t \not\equiv 0, -1, -2 \mod 2(p-1).
\end{cases}
$$
\end{cor}

\section{Mod $p^j$ congruences}\label{sec:congruences}

Let $p \ge 5$.
The congruence
$$ E_{p-1}(q) \equiv 1 \mod p $$
implies the congruence
\begin{equation}\label{eq:Ep-1cong}
E^{p^{j-1}}_{p-1}(q)  \equiv 1 \mod p^j. 
\end{equation}
It follows that multiplication by $E^{p^{j-1}}_{p-1}$ induces an injection
$$ \cdot E_{p-1}^{p^{j-1}} : M_t(\Gamma_0(N))_{\ZZ/p^j} \hookrightarrow
M_{t+(p-1)p^{j-1}}(\Gamma_0(N))_{\ZZ/p^j}. $$
(Here we regard $E_{p-1}$ as a modular form for $\Gamma_0(N)$.)
The image of this inclusion 
is characterized by the following theorem of Serre \cite[Cor.~4.4.2]{Katz}.

\begin{thm}[Serre]\label{thm:Serre}
Let $f_i$ be an elements of $M_{k_i}(\Gamma_0(N))_{\ZZ/p^j}$ for $i = 1,2$
and $k_1 < k_2$.
Then 
$$ f_1(q) = f_2(q) \in \ZZ/p^j[[q]] $$
if and only if
\begin{enumerate}
\item $k_1 \equiv k_2 \mod (p-1)p^{j-1}$, and
\item $f_2 = E_{p-1}^{\frac{k_2-k_1}{p-1}} f_1$.
\end{enumerate}
\end{thm}

\section{$M_2 Q(\ell)$: The $\beta$-family congruences}\label{sec:M2Ql}

Let $p \ge 5$, and let $\ell$ be a topological generator of $\ZZ_p^\times$.
In this section we prove Theorem~\ref{thm:main} and
Theorem~\ref{thm:converse}.  The key observation is the following.

\begin{lem}\label{lem:key}
The inclusion $\Gamma_\ell \hookrightarrow \MS_2$ induces 
an isomorphism
$$ H^0_c(\MS_2, \pi_t M_2 E_2)^{\Gal(\FF_p)} \xrightarrow{\cong}
H^0(\Gamma_\ell, \pi_t M_2 E_2)^{\Gal(\FF_p)}.
$$
\end{lem}

\begin{proof}
By Theorem~\ref{thm:dense}, the group $\Gamma_\ell$ is dense in $\MS_2$.
Since $\MS_2$ acts continuously on $\pi_t M_2E_2$, the invariants of
$\MS_2$ are the same as the invariants of $\Gamma_\ell$.
\end{proof}

Combined with Corollary~\ref{cor:collapseisomorphism4}, we have an
isomorphism
$$ H^0(C(\ell)^\bullet/(p^\infty, v_1^\infty))_t \cong H^0_c(\MS_2, \pi_t
M_2 E_2)^{\Gal(\FF_p)}. $$
The right-hand side has been computed by Miller-Ravenel-Wilson
\cite{MillerRavenelWilson}:

\begin{thm}[Miller-Ravenel-Wilson]\label{thm:MillerRavenelWilson}
The groups $H^0_c(\MS_2, \pi_*
M_2 E_2)^{\Gal(\FF_p)}$ are generated by elements
$$ \beta_{i/j,k} \in H^0_c(\MS_2, \pi_{2i(p^2-1) - 2j(p-1)}
M_2 E_2)^{\Gal(\FF_p)} $$
which generate cyclic summands of order $p^k$.  Here, for $i = sp^n$ with
$(s,p) = 1$, the indices $j$ and $k$ are taken subject to
\begin{enumerate}
\item $p^{k-1} \vert j$,
\item $j \le p^{n-k+1} + p^{n-k} - 1$,
\item either $j > p^{n-k} + p^{n-k-1}-1$ or $p^k \not \vert j$.
\end{enumerate}
\end{thm}

We now compute
$$ H^0(C^\bullet(\ell)/(p^\infty, v_1^\infty))_* $$
in terms of modular forms.

We have
$$ H^0(C^\bullet(\ell)/(p^\infty, v_1^\infty))_{2t} = 
\varinjlim_{k} \varinjlim_{\substack{j = sp^{k-1} \\ s \ge 0}} 
H^0(C^\bullet(\ell)/(p^k, v_1^j))_{2t+2j(p-1)}, $$
so it suffices to compute the latter.  
Proposition~\ref{prop:diffq}, Lemma~\ref{lem:MnTMF}, and
Theorem~\ref{thm:Serre} 
imply that, for $j \equiv 0 \mod
p^{k-1}$, we have:
\begin{align*}
\mc{B}_{t/j,k} & := H^0(C^\bullet(\ell)/(p^k, v_1^j))_{2t+2j(p-1)} \\
& = \ker \left(
\frac{M^0_{t+j(p-1)}}{(p^k, E_{p-1}^j)} \xrightarrow{d_0 - d_1}
\begin{array}{c}
\frac{M^0_{t+j(p-1)}}{(p^k, E_{p-1}^j)} \\
\oplus \\
\frac{M_{t+j(p-1)}(\Gamma_0(\ell))^0}{(p^k, E_{p-1}^j)}
\end{array}
\right)
\\ \\
& = \ker \left(
\frac{(M_{t+j(p-1)})^0_{\ZZ/p^k}}{(M_{t})_{\ZZ/p^k}^0} 
\xrightarrow{d_0 - d_1}
\begin{array}{c}
\frac{(M_{t+j(p-1)})^0_{\ZZ/p^k}}{(M_{t})^0_{\ZZ/p^k}} \\
\oplus \\
\frac{M_{t+j(p-1)}(\Gamma_0(\ell))_{\ZZ/p^k}^0}{M_{t}(\Gamma_0(\ell))_{\ZZ/p^k}^0}
\end{array}
\right)
\\ \\
& = \left\{ f \in 
\frac{(M_{t+j(p-1)})_{\ZZ/p^k}^0}{(M_{t})_{\ZZ/p^k}^0}
\: : \: 
\begin{array}{cl}
\mr{(i)} & (\ell^{t+j(p-1)} - 1)f(q) = g_1(q), \\
& \: \mr{for} \: g_1 \in
(M_{t})^0_{\ZZ/p^k} \\
\mr{(ii)} & \ell^{t+j(p-1)} f(q^\ell) - f(q) = g_2(q), \\
& \: \mr{for} \: g_2 \in M_t(\Gamma_0(\ell))^0_{\ZZ/p^k}
\end{array}
\right\}.
\end{align*}
Here, we are regarding the space of mod $p^k$ modular forms of 
weight $t$ as being embedding in the space of mod $p^k$ modular forms of
weight $t+j(p-1)$ through the inclusion induced by 
multiplication by $E^j_{p-1}$ using Theorem~\ref{thm:Serre}.

For a finitely generated abelian $p$-group $A$, we shall say that $a \in A$
is an \emph{additive generator of order $p^k$} if $a$ generates a cyclic 
subgroup of $A$ of order $p^k$.

\begin{thm}\label{thm:maindelta-1}
There is a one-to-one correspondence between the additive generators of
order $p^k$ in
$$ H^0(C(\ell)^\bullet/(p^\infty, v_1^\infty))_{2t} $$
and the modular forms $f \in M_{t+j(p-1)}^0$ for $j \equiv 0
\mod p^{k-1}$ satisfying
\begin{enumerate}
\item We have $t \equiv 0 \mod (p-1)p^{k-1}$.
\item The $q$-expansion $f(q)$ is not congruent to $0$ mod $p$.
\item We have $\ord_q f(q) > \frac{t}{12}$ or $\ord_q f(q) =
\frac{t-2}{12}$.
\item There does not exist a form $f' \in M_{t'}^0$ such that
$f'(q) \equiv f(q) \mod p^k$ for $t' < t + j(p-1)$.
\item[$\rm{(5)}_\ell$] There exists a form 
$$ g \in M_{t}(\Gamma_0(\ell))^0 $$
satisfying 
$$ f(q^\ell) - f(q) \equiv  g(q) \mod p^k. $$
\end{enumerate}
\end{thm}

We will need to make use of the following lemma.

\begin{lem}\label{lem:splitting}
There exist homomorphisms
$$ r_m: M_{t+m}^0 \rightarrow M_{t}^0 $$
such that, if 
$j \equiv 0 \mod p^{k-1}$,
the short exact sequences
$$ 0 \rightarrow (M_{t})_{\ZZ/p^k}^0 \xrightarrow{\cdot E^j_{p-1}} 
(M_{t+j(p-1)})_{\ZZ/p^k}^0 \rightarrow
\frac{(M_{t+j(p-1)})_{\ZZ/p^k}^0}{(M_{t})_{\ZZ/p^k}^0}
\rightarrow 0
$$
are split by the mod $p^k$ reduction of $r_{j(p-1)}$.
\end{lem}

\begin{proof}
Using the basis of Lemma~\ref{lem:basis}
we define explicit splitting morphisms
$$ r_m: M_{t+m} \rightarrow M_{t} $$
whose effect on basis vectors is given by
$$
r_m(\Delta^{n} e_{t+m-12n}) = 
\begin{cases}
\Delta^n e_{t-12n}, & t-12n = 0, \: \mr{or} \: t-12n = 2i \: \mr{for} \: 
i \ge 2, \\
0,  & \mr{otherwise}.
\end{cases}
$$
We just need to verify that $r_{j(p-1)}$ reduces to give the appropriate 
splittings.
By Condition~(3) of Lemma~\ref{lem:ek}, and (\ref{eq:Ep-1cong}), 
we have
$$ e_t(q)E^j_{p-1}(q) \equiv e_t(q) \equiv e_{t+j(p-1)}(q) \mod p^k. $$
We therefore compute
\begin{align*}
r_{j(p-1)}(\Delta^n e_{t - 12n}E^j_{p-1}) 
& \equiv r_{j(p-1)}(\Delta^n e_{t+j(p-1) - 12n}) \mod p^k \\
& \equiv \Delta^n e_{t - 12n} \mod p^k.
\end{align*}
\end{proof}

The splittings of Lemma~\ref{lem:splitting} induce splitting homomorphisms
which give short exact sequences
\begin{equation}\label{eq:splitting}
0 \leftarrow (M_{t})_{\ZZ/p^k}^0 \xleftarrow{r_{j,k}} 
(M_{t+j(p-1)})_{\ZZ/p^k}^0 \xleftarrow{\iota_{j,k}}
\frac{(M_{t+j(p-1)})_{\ZZ/p^k}^0}{(M_{t})_{\ZZ/p^k}^0}
\leftarrow 0
\end{equation}
which are compatible as $k$ and $j$ vary.

\begin{lem}\label{lem:imageijk}
For $t$ even, 
the image of the homomorphism $\iota_{j,k}$ is given by
$$
\{ f \in (M_{t+j(p-1)})_{\ZZ/p^k}^0 \: : \: \ord_q f(q) >
\frac{t}{12} \: \mr{or} \: \ord_q f(q) = \frac{t-2}{12} \}.
$$
\end{lem}

\begin{proof}
A basis of $(M_t)^0_{\ZZ/p^k}$ is given by
$$ \{ \Delta^n e_{t-12n} \: : \: n \le \frac{t}{12}, \: n \ne
\frac{t-2}{12} \}. $$
The image of this basis under $\iota_{j,k}$ is spanned by
$$ \{ \Delta^n e_{t-12n+j(p-1)} \: : \: n \le \frac{t}{12}, \: n \ne
\frac{t-2}{12} \}. $$
Since
$$ \Delta^n e_{t-12n+j(p-1)} = q^n + \cdots $$
we deduce the result.
\end{proof}

\begin{proof}[Proof of Theorem~\ref{thm:maindelta-1}]
Suppose that $b' \in \mc{B}_{t/j,k'}$ is an additive generator of order
$p^k$. Let $f'$ be the lift
$$ f' := \iota_{j,k'} (b') \in
(M_{t+j(p-1)})_{\ZZ/p^{k'}}^0. $$
Since $b'$ is assumed to be an
additive generator or order $p^k$ and $\iota_{j,k'}$ is injective, 
we deduce that $f'$ is a modular form in
$(M_{t+j(p-1)})_{\ZZ/p^{k'}}$ of exact order $p^k$.  Hence $f' = p^{k'-k}f$
for some modular form $f \in (M_{t+j(p-1)})_{\ZZ/p^k}^0$.  It is
simple to check that the image 
$$ b \in
\frac{(M_{t+j(p-1)})_{\ZZ/p^k}^0}{(M_t)_{\ZZ/p^k}^0} $$
represents an element of $\mc{B}_{t/j,k}$.

It follows that 
the additive generators of order $p^k$ in
$$ H^0(C(\ell)^\bullet/(p^\infty, v_1^\infty))_t = \colim_{k'}
\colim_{\substack{j =  sp^{k'-1} \\ s \ge 1}} \mc{B}_{t/j,k'} $$
exactly correspond to the additive generators of order $p^k$ in 
$\mc{B}_{t/j,k}$ 
which are not in the image of the inclusion
$$ \cdot E^{p^{k-1}}_{p-1} : \mc{B}_{t/j-p^{k-1},k} \hookrightarrow
\mc{B}_{t/j,k}. $$
Suppose that $b$ is such an additive generator.  Let $f$ be the lift
$$ f := \iota_{j,k} (b) \in
(M_{t+j(p-1)})_{\ZZ/p^{k}}^0. $$
Then by Lemma~\ref{lem:imageijk} the lift $f$ satisfies
$$ \ord_q f(q) > \frac{t}{12} \: \mr{or} \: \ord_q f(q) = \frac{t-2}{12}.
$$
From the definition of $\mc{B}_{t/j,k}$ we have 
$$
\begin{array}{cl}
\mr{(i)} & (\ell^{t+j(p-1)} - 1)f(q) = g_1(q), \\
& \: \mr{for} \: g_1 \in
(M_{t})^0_{\ZZ/p^k} \\
\mr{(ii)} & \ell^{t+j(p-1)} f(q^\ell) - f(q) = g_2(q), \\
& \: \mr{for} \: g_2 \in M_t(\Gamma_0(\ell))^0_{\ZZ/p^k}
\end{array}
$$
Since $j \equiv 0 \mod p^{k-1}$ we deduce that
$$ \ell^{t+j(p-1)} \equiv \ell^t. $$
Let $v = \nu_p(\ell^t - 1)$.
Condition (i) above implies that 
$$ f(q) \equiv \frac{g_1(q)}{\ell^t-1} \mod p^{k-v}. $$
But, if $b'' \in \mc{B}_{t/j,k-v} $ is the image of the mod $p^{k-v}$ reduction 
of $b$, then 
$$ f(q) \equiv \iota_{j,k-v} (b'') \mod p^{k-v} $$
and thus, by the exactness of (\ref{eq:splitting}), we have
$$ r_{j, k-v}( f) = g_1 = 0. $$
Thus we actually have
$$ (\ell^t -1)f(q) \equiv 0 \mod p^k. $$
Since $f(q)$ has order $p^k$, we deduce that
$$ \ell^t  \equiv 1 \mod p^k. $$
Since $\ell$ is a topological generator of $\ZZ_p^\times$, we deduce that
$$ t \equiv 0 \mod (p-1)p^{k-1}. $$
Thus condition (ii) may be rewritten as
$$ f(q^\ell) - f(q) = g_2(q) \: \mr{for} \: 
g_2 \in M_t(\Gamma_0(\ell))^0_{\ZZ/p^k}. $$
We have therefore verified conditions (1)--(5) of
Theorem~\ref{thm:maindelta-1}.

For the converse direction, suppose $f \in (M_{t+j(p-1)})_{\ZZ/p^k}$ 
satisfies conditions (1)--(5) of
Theorem~\ref{thm:maindelta-1}.
Then by Lemma~\ref{lem:imageijk}, $f$ is in the image of $\iota_{j,k}$.
Consider the image
$$ b = [f] \in \frac{(M_{t+j(p-1)})^0_{\ZZ/p^k}}{(M_t)^0_{\ZZ/p^k}}.
$$
of $f$ in the quotient.
Observe that by (2), the element $b$ has order $p^k$.  We just need to
verify that it is an element of $\mc{B}_{t/j,k}$, which amounts to seeing
that $f$ satisfies conditions (i) and (ii) above.  But condition~(1)
implies that
$$ \ell^t \equiv \ell^{t+j(p-1)} \equiv 1 \mod p^k. $$
This immediately implies that $f$ satisfies condition (i).  Condition~(ii)
then follows from condition (5).
\end{proof}

Observe that if $S$ is a set of primes which contains $\ell$ and does not
contain $p$, then we have
$$ \Gamma_\ell \subseteq \Gamma_S \subset \MS_2. $$
Since $\Gamma_\ell$ is dense in $\MS_2$, the subgroup $\Gamma_S$ is dense
in $\MS_2$.  We therefore deduce the following lemma.

\begin{lem}
For a set of primes $S$ not containing $p$ and containing $\ell$, there is
an isomorphism
$$ H^0(\Gamma_S, \pi_t M_2E_2)^{\Gal(\FF_p)} \xrightarrow{\cong}
H^0(\Gamma_\ell, \pi_t M_2E_2)^{\Gal(\FF_p)}. $$
\end{lem}

In particular, letting $\ell'$ be a prime in $S$, we have
a zig-zag
$$ H^0(\Gamma_\ell, \pi_t M_2E_2)^{\Gal(\FF_p)} \xleftarrow{\cong}
H^0(\Gamma_S, \pi_t M_2 E_2)^{\Gal(\FF_p)} \hookrightarrow 
H^0(\Gamma_{\ell'}, \pi_t M_2 E_2)^{\Gal(\FF_p)}. $$
If $\ell'$ also generates $\ZZ_p^\times$, then the inclusion is an
isomorphism.
Corollary~\ref{cor:collapseisomorphism4} allows us to deduce:

\begin{cor}\label{cor:rigid}
For any prime $\ell' \ne \ell$
There is an
inclusion
$$ H^0(C(\ell)^\bullet/(p^\infty, v_1^\infty))_t \hookrightarrow
H^0(C(\ell')^\bullet/(p^\infty, v_1^\infty))_t. $$
If $f$ satisfies Conditions (1)-(4) and $\rm{(5)}_\ell$ of
Theorem~\ref{thm:maindelta-1}, then it satisfies condition
$\rm{(5)}_{\ell'}$.
\end{cor}

We finish this section by observing that the results of this section
combine to give proofs of some of the theorems stated in
Section~\ref{sec:intro}

\begin{proof}[Proofs of Theorems~\ref{thm:main}, \ref{thm:converse},
\ref{thm:rigid}]
Corollary~\ref{cor:rigid} implies Theorem~\ref{thm:rigid}.  
The element
$$ \beta_{i/j,k} \in H^0(\MS_2, \pi_{2i(p^2-1) - 2j(p-1)} M_2
E_2)^{\Gal(\FF_p)} $$
detects a corresponding Greek letter element
$$ \beta_{i/j,k} \in \Ext^{2,*}_{BP_*BP}(BP_*, BP_*) $$
in the chromatic spectral sequence if $i > 0$ and $i \ne p^n$
\cite{MillerRavenelWilson} (if $ i = p^n$, then $j$ must be greater than
or equal to $p^n$). Thus Theorems~\ref{thm:main} and \ref{thm:converse}
follow from Theorem~\ref{thm:maindelta-1}.  Note that the modular forms $f =
f_{i/j,k}$ of Theorems~\ref{thm:main} and \ref{thm:converse} are taken to
be holomorphic at the cusps, whereas in Theorem~\ref{thm:maindelta-1}, they
are merely assumed to be meromorphic at the cusps.  This discrepancy is
resolved by noting that if $i, j, k$ are chosen such that $\beta_{i/j,k}$
exists in $\Ext_{BP_*BP}(BP_*, BP_*)$, then 
$$ t = i(p^2-1) - j(p-1) \ge 0. $$
Therefore,
condition~(2) of Theorem~\ref{thm:maindelta-1} guarantees that the modular
forms in question are holomorphic at the cusps.
\end{proof}

\section{Greek letter elements in the Hurewicz image of
$Q(\ell)$}\label{sec:greek}

Since the cosimplicial spectrum $Q(\ell)^\bullet$ is a cosimplicial object
in the category of $E_\infty$-ring spectra, 
the equivalence $Q(\ell) \simeq \Tot Q(\ell)^\bullet$
(\ref{eq:cosimplicial}) allows us to regard $Q(\ell)$ as an $E_\infty$-ring
spectrum.  In particular, it possesses a unit map
$$ S \rightarrow Q(\ell) $$
which, by Lemma~\ref{lem:E2local}, localizes to give a map
$$ S_{E(2)} \rightarrow Q(\ell). $$
In this section we prove:

\begin{thm}\label{thm:greek}
The images of the elements $\alpha_{i/j}$ and the elements $\beta_{i/j,k}$ 
under the homomorphism
$$ \pi_* S_{E(2)} \rightarrow \pi_*Q(\ell) $$
are non-trivial.
\end{thm}

We first will need a lemma.

\begin{lem}\label{lem:e}
The map
$$ \pi_tM_1S \rightarrow \pi_t M_1Q(\ell) $$
is an isomorphism for $t \equiv 0 \mod 4$.
\end{lem}

\begin{proof}
Let $C$ be an ordinary elliptic curve over $\br{\FF}_p$, so that there is
an isomorphism of formal groups
$$ C^\wedge \cong \widehat{\GG}_m. $$
Let $U$ be the formal
neighborhood of the associated point of $\mc{M}(K^p_0)$, which carries a
universal deformation $\td{C}/U$ of $C$.  
Let 
$$ E = \mc{E}_{K^p_0}(U) $$
be the sections of the sheaf $\mc{E}$ over $U$.  By Serre-Tate theory, and
the deformation theory of $p$-divisible groups \cite[7.1]{BehrensLawson}, 
we deduce that
$$ U \cong \mr{Spf}(W(\bar{\FF}_p)[[x]]) $$
and therefore that $E$ is an even periodic ring spectrum with $\pi_0(E)
\cong W(\bar{\FF}_p)[[x]]$, with associated formal group given by
$\td{C}^\wedge$.  
The cofiber
$$ E \xrightarrow{\cdot x} E \rightarrow E/x $$
is an even periodic ring spectrum \cite{EKMM}.  The restriction $\td{C}_{\mr{can}}$ 
of the deformation
$\td{C}$ to $\pi_0(E/x) \cong W(\bar{\FF}_p)$ is the canonical deformation
of $C$ (the deformation whose $p$-divisible group splits).  The formal group $\td{C}^\wedge_{\mr{can}}$ is therefore a universal
deformation of $\GG_m/\bar{\FF}_p$, and we conclude that there is an
isomorphism
$$ \td{C}^\wedge_{\mr{can}} \cong \widehat{\GG}_m $$
between the formal group for $E/x$ and the multiplicative formal group.  In
particular, this implies that there is an equivalence of ring spectra
$$ K_p \otimes_{\ZZ_p} W(\bar{\FF}_p) \cong E/x, $$
where $K_p$ is the $p$-adic $K$-theory spectrum.
Now, the unit map $S \rightarrow K_p$ induces an inclusion
$$ \pi_{2t} M_1 S \hookrightarrow \pi_{2t} M_1 K_p $$
(it gives the Adams $e$-invariant).  Therefore the unit map for $E/x$
induces an inclusion
$$ \pi_{2t} M_1 S \hookrightarrow \pi_{2t} M_1 (E/x). $$
However, the unit for $E/x$ is homotopic to the composite
$$ S \rightarrow Q(\ell) \simeq \Tot Q(\ell)^\bullet \rightarrow Q(\ell)^0
= \TMF_p \rightarrow E \rightarrow E/x $$
because all of the maps in the composite are maps of ring spectra.
We deduce that the maps
$$ \pi_{t}M_1 S \rightarrow \pi_t M_1 Q(\ell) $$
are injective for $t \equiv 0 \mod 2$.  By Corollary~\ref{cor:M1Ql}, these
(finite) groups are abstractly isomorphic for $t \equiv 0 \mod 4$ and $t
\ne
0$.  The result for $t \ne 0$ therefore is proven.  The cases of $t = 0$
follows immediately from the fact that the map $\pi_0 S \rightarrow \pi_0
Q(\ell) \cong \ZZ_p$ is a
map of rings.
\end{proof}

\begin{proof}[Proof of Theorem~\ref{thm:greek}]
Consider the map of chromatic spectral sequences:
$$
\xymatrix{
\pi_k M_n S_{E(2)} \ar@{=>}[r] \ar[d] &
\pi_{k-n} S_{E(2)} \ar[d] 
\\
\pi_k M_n Q(\ell) \ar@{=>}[r] &
\pi_{k-n} Q(\ell)  
}
$$
The elements $\alpha_{i/j} \in \pi_{2i(p-1)}(M_1 S)$ are known to be
permanent cycles for $i > 0$, and therefore map to permanent cycles in the
chromatic spectral sequence for $Q(\ell)$.  
By Lemma~\ref{lem:e}, the images of $\alpha_{i/j}$ in $\pi_t M_1 S$ are
nontrivial, and generate these groups for $t \equiv 0 \mod 4$.
Since, by Corollary~\ref{cor:M0Ql}, $\pi_t M_0 Q(\ell)$ is zero for
$t \ne 0, -1, -2$, there are no non-trivial differentials
$$ d_1: \pi_t M_0 Q(\ell) \rightarrow \pi_t M_1 Q(\ell) $$
for $t > 0$.  We deduce that the images of the elements $\alpha_{i/j}$ in
the chromatic spectral sequence for $Q(\ell)$ are non-trivial permanent
cycles, and hence witness the non-triviality of the images of the elements
$\alpha_{i/j}$ in $\pi_* Q(\ell)$.  As a side-effect, we have also
determined that the groups $\pi_t M_1 Q(\ell)$ are generated by permanent
cycles for $t \equiv 0 \mod 4$.  That, combined with the fact that $\pi_t
M_0 Q(\ell)$ is zero for $t$ positive, allows us to deduce that there are
no non-trivial differentials killing elements of $\pi_t M_2 Q(\ell)$ for $t
\equiv 0 \mod 4$.  To complete the proof of the theorem, it suffices to
show that the images of the elements $\beta_{i/j,k}$ are non-trivial under
the homomorphism
\begin{equation}\label{eq:M2hur}
\pi_{t} M_2 S \rightarrow \pi_{t}
M_2 Q(\ell)
\end{equation}
where $t = 2i(p^2-1) - 2j(p-1)$.  But, for such $t$, the map
(\ref{eq:M2hur}) is given by the composite of isomorphisms
\begin{align*}
\pi_t M_2 S & 
\cong H^0_c(\MS_2, \pi_t M_2 E_2)^{\Gal(\FF_p)} \\
& \cong H^0(\Gamma_\ell, \pi_t M_2 E_2)^{\Gal(\FF_p)} \\
& \cong \pi_t M_2 Q(\ell)
\end{align*}
given by Corollary~\ref{cor:collapseisomorphism1}, Lemma~\ref{lem:key}, and
Lemma~\ref{lem:collapseisomorphism3}.
\end{proof}

\nocite{*}
\bibliography{beta}

\end{document}